\documentclass[11pt]{article} 

\usepackage{amsmath, amsfonts, amssymb}
\usepackage{mathrsfs}
\usepackage{theorem}
\usepackage{bm}
\usepackage[dvipdf]{graphicx}

\RequirePackage[colorlinks,citecolor=blue,urlcolor=blue,bookmarksopen]{hyperref}

\newtheorem{mythm}{Theorem}[section]

\newtheorem{mylem}[mythm]{Lemma}

{\newtheorem{myrem}[mythm]{Remark}}
{
\newtheorem{myexam}[mythm]{Example}}%

\def\R{\mathbb R}

\def\N{\mathbb N}
\def\C{\mathscr C}

\def\F{\mathscr F}

\def\E{\mathbb E}
\def\p{\mathbb P}
\def\e{\text{\rm{e}}}

\def\d{\text{\rm{d}}}

\def\La{\Lambda}
\def\veps{\varepsilon}

\def\S{\mathcal S}

\def\wt{\widetilde}

\def\W{\mathbb{W}}

\newenvironment{proof}{{\noindent\it Proof.}\ }{\hfill $\square$\par}

\numberwithin{equation}{section}

\allowdisplaybreaks

\begin{document}

\title{Ergodicity and stability of hybrid systems with piecewise constant type state-dependent switching \footnote{Supported in
 part by National Key R\&D Program of China (No. 2022YFA1000033) and NNSFs of China (No. 12271397, 11831014,12101186)}}

\author{ Jinghai Shao$^a$\thanks{a: Center for Applied Mathematics, Tianjin University, Tianjin 300072, China.
}, Lingdi Wang$^b$\thanks{b: School of Mathematics and Statistics, Henan University, Kaifeng 475001, China.}, Qiong Wu$^a$}
\date{}
\maketitle

\begin{abstract}
  To deal with stochastic hybrid systems with general state-dependent switching, we propose an approximation method  by a sequence of stochastic hybrid systems with piecewise constant type switching. The convergence rate in  the Wasserstein distance is estimated in terms of the difference between transition rate matrices. Our method is based on an elaborate construction of coupling processes in terms of Skorokhod's representation theorem for jumping processes.  Moreover, we establish explicit criteria on the ergodicity and stability for stochastic hybrid systems with piecewise constant type switching. Some examples are given to illustrate the sharpness of these criteria.
\end{abstract}



\section{Introduction}

Stochastic hybrid systems can model the interaction between continuous dynamics and discrete dynamics, and because of their versatility, they have been widely used as effective models for capturing the intricacies of complex systems; see, for instance, applications in mathematical finance \cite{GZ04,SC09}, in biology \cite{Cru,Font}, in biochemistry \cite{RRK} and references therein. These processes are also called stochastic processes with regime-switching. See the monograph \cite{YZ} for more introduction on various applications of such models. Much effort has been devoted to the study of long time behavior of such processes. We refer the reader to \cite{PP,Sh15a} for the exponential ergodicity in the total variance distance, and to \cite{CH,Sh15b} in the Wasserstein distance; to \cite{BBG,MY,SX,YZ} for the stability in various sense; to \cite{Bar,LS20,SY05} for the characterization of the invariant probability measures.

Due to the intensive interaction between the continuous component and the discrete component, it is a  very challenging task to characterize the long time behavior of stochastic hybrid systems with state-dependent regime-switching, and considerable effort has been devoted to this topic.   In the existing literatures, two kinds of methods have been developed to deal with the state-dependent regime-switching processes. One is to construct suitable coupling process to control the state-dependent regime-switching process with a state-independent one. This method has been developed in \cite{CH,MT,Sh18,Sh22,SX19}. Another method is to construct directly the desired Lyapunov function by viewing the whole system as a Markov process; see, for example, the monograph \cite{YZ} and the recent works \cite{Sh15a,SX} based on $M$-matrix theory. Generally, these two methods could provide certain verifiable conditions at the cost of sharpness. In this work, we shall develop an alternative method: approximation method.

The basic idea of approximation method is:  instead of considering directly a stochastic hybrid system with a general state-dependent switching process, we approximate it with a sequence of stochastic hybrid  systems with piecewise constant type state-dependent switching. Heuristically, we want to approximate a general continuous matrix-valued function with a sequence of step functions, i.e. piecewise constant matrix-valued functions.  As a special class of state-dependent regime-switching processes, we can provide a relatively complete characterization on the ergodicity and stability for stochastic hybrid systems with piecewise constant type switching. Furthemore, the approximation is measured using the Wasserstein distance between the distributions of two processes, and the convergence rate can also be estimated in terms of the distance between the transition rate matrices. Meanwhile, it is also necessary to study stochastic hybrid systems with piecewise constant type switching from the application point of view. For instance, in \cite{RRK}, the authors proposed a stochastic  hybrid system to model the biodiesel reaction. There, the switching will happen  once the continuous component leaves certain invariant regions.


Precisely, give a filtered probability space $(\Omega,\F,\p)$ with filtration $\{\F_t\}_{t\geq 0}$. Consider a stochastic hybrid system $(X_t,\La_t)_{t\geq 0}$  as follows: the continuous component $(X_t)$ satisfies the SDE
\begin{equation}\label{a-1}
\d X_t=b(X_t,\La_t)\d t +\sigma(X_t,\La_t)\d B_t,
\end{equation} where $(B_t)_{t\geq 0}$ is a $d$-dimensional $\F_t$-adapted Brownian motion. The discrete component $(\La_t)_{t\geq 0}$ is a jumping process on a finite state space $\S=\{1,2,\ldots,N\}$ satisfying
\begin{equation}\label{a-2}
\p(\La_{t+\delta}=j|\La_t=i,X_t=x)=\begin{cases}
q_{ij}(x)\delta+o(\delta), &i\neq j,\\
1+q_{ii}(x)\delta+o(\delta), &i=j,
\end{cases}
\end{equation} for $\delta>0$. For each $x\in \R^d$, $(q_{ij}(x))_{i,j\in\S}$ is a $Q$-matrix. If $(q_{ij}(x))_{i,j\in \S}$ is a $Q$-matrix independent of $x$, then $(X_t,\La_t)_{t\geq 0}$ is  called a Markovian  stochastic hybrid system or a Markovian regime-switching process.

In the existing literatures on the state-dependent hybrid systems, it is often assumed a priori that the state-dependent switching function $x\mapsto q_{ij}(x)$ is a continuous function for every $i,j\in \S$. See, for instance, \cite{GAM,NY18,Sh15}, monograph \cite{YZ}.  Step function is a widely used substitute or approximation of continuous function in many research field. However, the study on the hybrid system  with  switching rates being  a step function in $x$ is very limited.   Owing to the demand of realistic application,   the previous model has been extended to include the impact of the history of the process $(X_t)$, so that $q_{ij}(\cdot)$ becomes a functional on the path space $C([-\tau, 0];\R^d)$; see, for example,  \cite{NY16} and \cite{SZ21}.

In this work, we are interested in the following kind of discontinuous switching functions in the form:
\begin{equation}\label{a-3}
q_{ij}(x)=\sum_{k=1}^{m+1} q_{ij}^{(k)}\mathbf1_{[\alpha_{k-1},\alpha_k)} (|x|),\quad \ i,j\in\S, \ x\in \R^d,
\end{equation} where  $ \Delta_{m}:=\{0=\alpha_0< \alpha_1<\ldots<\alpha_{m}<\alpha_{m+1}\!=\!\infty \}$ for some $m\in \N$ is a finite division of $[0, \infty)$, and $(q_{ij}^{(k)})_{i,j\in\S}$ are  $Q$-matrices for $k=1,\ldots,m+1$.
When $\R^d=\R$, besides \eqref{a-3}, we also consider a more general type of switching function:
\begin{equation}\label{a-4}
q_{ij}(x)=q_{ij}^{(0)} \mathbf1_{(-\infty,\alpha_0)}(x) + \sum_{k=1}^{m+1} q_{ij}^{(k)}\mathbf1_{[\alpha_{k-1},\alpha_k)}(x) , \quad i,j\in\S, \, x\in \R,
\end{equation}
where $(q_{ij}^{(k)})$ are $Q$-matrices on $\S$, $m\in \N$, and
\[\{-\infty<\alpha_{0}<\alpha_{1}<\ldots  <\alpha_{m}<\alpha_{m+1} =\infty\} \]
is a finite partition of $\R$.

We call a stochastic hybrid system $(X_t,\La_t)_{t\geq 0}$ with $(\La_t)_{t\geq 0}$ satisfying \eqref{a-2}, \eqref{a-3} or \eqref{a-4} when $d=1$ a \emph{stochastic hybrid system with piecewise constant type switching}.

The main reason to study the switching function  in the form \eqref{a-3} is its simplicity, which will be seen from our results on the egodicity and stability of such kind of hybrid system, compared especially with those for hybrid systems with a general state-dependent regime-switching. Just because of its simplicity, such functions are widely used in various research fields. For instance, in the control theory, the bang-bang control as the simplest control algorithm is  widely used in many types of industrial control systems. See, e.g. \cite{MS} for the bang-bang principle for the linear control system. Moreover, the optimal control policy in the stochastic control problem in \cite{HT83} is a bang-bang policy. In the study of particle system,   Cox and Durrett \cite{CD92} discovered that certain nonlinear voter models can coexist even in one dimension. Among those of greatest interest are the threshold voter models (cf. e.g. \cite{Lig}).  See \cite{Du95} for more threshold models and \cite{DL94} for the biological applications of these models of interacting particle system.

We are mainly concerned with the following three problems on the stochastic hybrid system with piecewise constant type switching in this work: 1) The wellposedness of stochastic hybrid systems with piecewise constant type switching; 2)  If the sequence of piecewise constant type $Q$-matrices $(q_{ij}^n(x))$ approximates to a $Q$-matrix $(q_{ij}(x))$ which is continuous in $x$, the approximation problem of the corresponding stochastic hybrid systems $(X_t^n)$  to $(X_t)$; 3) Stability and recurrent property of stochastic hybrid system with piecewise constant type switching. Accordingly, the contribution of this work consists of the following three topics.
\begin{enumerate}
  \item Provide conditions to ensure the wellposedness of the stochastic hybrid system $(X_t,\La_t)_{t\geq 0}$ with piecewise constant type switching.  To this end, we need to generalize the Skorokhod representation theorem to deal with the non-continuity of  $x\mapsto q_{ij}(x)$, cf., for instance, \cite{GAM}, \cite{Sko}, \cite{Sh15}, \cite{YZ}, when $x\mapsto q_{ij}(x)$ is continuous.

  \item If we use a sequence of state-dependent $Q$-matrix in the form \eqref{a-3} to approximate a state-dependent $Q$-matrix $(q_{ij}(x))$ satisfying $x\mapsto q_{ij}(x)$ being Lipschitz continuous, the corresponding hybrid systems will converge to the limit system.

  \item  Provide explicit conditions to justify  the stability in probability and ergodicity of hybrid systems with piecewise constant type switching. These conditions generalize the corresponding results for the hybrid systems with Markovian switching, and are quite sharp as being illustrated via concrete examples.
\end{enumerate}

The paper is organized as follows. In Section 2, we establish the existence and uniqueness of hybrid systems with piecewise constant type switching by developing Skorokhod's representation theorem. Section 3 is devoted to the approximation problem. We shall show that when the sequence of $Q$-matrices $(q_{ij}^{(n)}(x))$, $n\geq 1$, approximates to a $Q$-matrix $(q_{ij}(x))$ being Lipschitz continuous in $x$, then the corresponding processes $X_t^{(n)}$ will approximate $X_t$ in certain sense.  In Section 4, we investigate the stability in probability and ergodic property of stochastic hybrid systems with piecewise constant type switching. The sharpness of the obtained criteria is illustrated through examples.

\section{Wellposedness of hybrid system with piecewise constant type switching}

This section is devoted to the existence and uniqueness of hybrid systems with piecewise constant type switching. We shall first establish the existence and uniqueness of strong solution to some jump-diffusion differential equation, then show that the solution satisfies SDEs \eqref{a-1}, \eqref{a-2} and \eqref{a-3}. This generalizes the classical Skorokhold representation theorem from the continuity setting to a setting including certain special discontinuity.

Consider the hybrid system $(X_t,\La_t)$ satisfying \eqref{a-1}, \eqref{a-2} and \eqref{a-3}.
Let $b:\R^d\times\S\to \R^d$, $\sigma:\R^d\times \S\to \R^{d\times d}$, and $(q_{ij}(x))_{i,j\in\S}$ be a conservative $Q$-matrix on $\S$ for each $x\in \R^d$.

\noindent\textbf{Assumption A}:
\begin{itemize}
  \item[$\mathrm{(A1)}$] There exists $K_1>0$ such that
      \[|b(x,i)-b(y,i)|+|\sigma(x,i)-\sigma(y,i)|\leq K_1|x-y|,\quad x,y\in \R^d, \ i\in\S.\]

  \item[$\mathrm{(A2)}$] For each $k\geq 0$, the $Q$-matrix $(q_{ij}^{(k)})_{i,j\in\S}$ in \eqref{a-3} or \eqref{a-4} is irreducible and conservative, which means that $q_i^{(k)}=-q_{ii}^{(k)}=\sum_{j\neq i}q_{ij}^{(k)}$ for every $i\in\S$.
\end{itemize}
\begin{myrem}\label{rem-1}
Condition  (A1) is used to ensure the existence and uniqueness of the solution to \eqref{a-1} as usual, which can be weakened to be non-Lipschitz as in \cite{Sh15}, or be in certain integrable space as in  \cite{Zh19}.
Condition (A2) is a standard condition in the study of continuous time Markov chains.
\end{myrem}

Let
\begin{equation}\label{K0}
K_0=\max\big\{q_{i}^{(k)};\ i\in \S, 1\leq k\leq m\!+\!1\big\}.
\end{equation}
In the spirit of Skorokhod's representation theorem \cite{Sko}, we shall construct a sequence of intervals $\Gamma_{ij}(x)$ whose length is related to $q_{ij}(x)$, and then use them to define a stochastic process in terms of a Poisson random measure, which will be our desired piecewise constant type switching process.
Precisely, put
\[\Gamma_{1k}(x)=[(k-2)K_0,(k-2)K_0 +q_{1k}(x)),\quad 2\leq k\leq N,\ \  U_1=[0,NK_0).
\]
By the definition of $K_0$ in \eqref{K0} and \eqref{a-3}, it is clear that
\[\Gamma_{1k}(x)\bigcap \Gamma_{1j}(x)=\emptyset,\quad k\neq j.\]
For $n\geq 2$,  put
\begin{equation}\label{g-1}
\begin{split}
  \Gamma_{nk}(x)&\!=\big[2(n\!-\!1)NK_0\!-\!( n\!-\!k)K_0,2(n\!-\!1)NK_0\!-\!(n\!-\!k) K_0 \!+\! q_{nk}(x)\big) \ \text{for} \ k<n,\\
  \Gamma_{nk}(x)&\!=\big[2(n\!-\!1)NK_0\!+\! (k\!-\!n \!-\!1)K_0, 2(n\!-\!1)NK_0\!+\!(k\! -\!n\!-\!1)K_0 \!+\! q_{nk}(x)\big)\ \text{for}\ k>n.
\end{split}
\end{equation}
Denote $U_n=\big[(2n-3)NK_0,(2n-1)NK_0)$, $n\geq 2$. Then $\Gamma_{nk}(x)\subset U_n$.
Let
\begin{equation}\label{g-1.5}
\kappa_0=(2N-1)NK_0,
\end{equation} then $U_n\subset [0,\kappa_0]$ for $1\leq n\leq N$.   For the convenience of notation, we put $\Gamma_{ii}(x)=\emptyset$ and $\Gamma_{ij}(x)=\emptyset$ if $q_{ij}(x)=0$.
Let
\begin{equation}\label{g-2}
\vartheta(x,i,z)=\begin{cases} j-i, &\text{if $z\in \Gamma_{ij}(x)$},\\
0, &\text{otherwise}.
\end{cases}
\end{equation}
Let us consider the SDEs
\begin{align}\label{b-1}
  \d \wt X_t&=b(\wt X_t,\wt \La_t)\d t+\sigma(\wt X_t,\wt \La_t)\d B_t,\\ \label{b-2}
  \d \wt \La_t&=\int_{[0,\kappa_0]}\!\vartheta(\wt X_t,\wt\La_{t-},z)\mathcal{N}(\d t,\d z)
\end{align} with initial value $(\wt X_0,\wt \La_0)=(x_0,i_0)\in \R^d\times\S$, where $(B_t)_{t\geq 0}$ is a $d$-dimensional Brownian motion; $\mathcal{N}(\d t,\d z)$ is a Poisson random measure over $[0,\kappa_0]$ with intensity measure $\d t\!\times\!\d z$ and is independent of $(B_t)_{t\geq 0}$.

\begin{mythm}\label{thm-1}
Assume (A1)  and (A2) hold. Then the system of SDEs \eqref{b-1}, \eqref{b-2} admits a pathwise unique strong solution $(\wt X_t,\wt \La_t)_{t\geq 0}$ for every initial value $(x_0,i_0)\in\R^d\times\S$.

Assume, in addition, that for every $t\geq 0$, $\p(|\wt X_t| =\alpha_k)=0$ for   $k=0,\ldots,m$, then $(\wt X_t,\wt \La_t)$ is a solution to  \eqref{a-1}, \eqref{a-2} with  $(q_{ij}(x))$ satisfying \eqref{a-3}.
\end{mythm}

\begin{proof}
First, let us show the wellposedness of SDEs \eqref{b-1}, \eqref{b-2}. To be more precise, let us introduce a product probability space to emphasize the independence between $(B_t)_{t\geq 0}$ and $\mathcal{N}(\d t,\d z)$ in \eqref{b-1}, \eqref{b-2}.

Let $(\Omega_i,\F^i,\p_i)$, $i=1,2$, be two probability space, and
\[\Omega=\Omega_1\times\Omega_2,\ \F=\F^1\times \F^2,\ \p=\p_1\times\p_2.\]We use $\omega=(\omega_1,\omega_2)$ to denote a generic element in $\Omega$. In the following, we suppose that $(B_t)_{t\geq 0}$ is a Brownian motion defined on $(\Omega_1,\F^1,\p_1)$, and $\mathcal{N}(\d t,\d z)$ on $(\Omega_2,\F^2,\p_2)$. In the sequel, the mutually independent processes $\Omega\ni (\omega_1,\omega_2)=\omega\mapsto  B_t(\omega_1)$, $\Omega\ni \omega\mapsto\mathcal{N}(\d t,\d z)(\omega_2)$ are used.

Let $0<\zeta_1<\zeta_2<\cdots<\zeta_n<\cdots$ be the sequence of jumping times of the Poisson point process $(p(t))_{t\geq 0}$ associated with the Poisson random measure $\mathcal{N}(\d t,\d z)$, which depends only on $\omega_2\in \Omega_2$. Since its intensity measure supported in $[0,\kappa_0]$ is a finite measure, $(p(t))_{t\geq 0}$ has finite number of jumps in every finite time interval $[0,T]$, $T>0$, and so
\[\lim_{n\to \infty}\zeta_n(\omega_2)=\infty,\quad \text{$\p_2$-a.e. $\omega_2$}.\]
By virtue of \eqref{b-2}, $\wt \La_t$ can jump only at some $\zeta_n$. We will use the independence between $\mathcal{N}(\d t,\d z)$ and $(B_t)_{t\geq 0}$ to construct the solution piece by piece.

For each $i\in\S$, consider the SDE
\begin{equation}\label{b-5}
\d X_t^{i}(\omega)= b(X_t^{i}(\omega),i)\d t+\sigma(X_t^{i}(\omega),i)\d B_t(\omega_1).
\end{equation}
Under condition (A1), it is well known that SDE \eqref{b-5} admits a unique strong solution for every initial value $X_s=\xi$ with $s>0$ and the random variable  $\xi$ on $\Omega$. So, given $i_0\in\S$,
we can define \[ \wt X_t(\omega_1,\omega_2)=X_t^{i_0}(\omega_1) \  \text{for $t\in [0,\zeta_1(\omega_2)]$ and $\wt \La_t(\omega_2)=i_0$ for $t\in [0,\zeta_1(\omega_2))$},\] where $X_t^{i_0}$ is the unique solution to \eqref{b-5} with initial value $X_0^{i_0}=x_0$.
Furthermore, if $p(\zeta_1)(\omega)\in \Gamma_{i_0j}(\wt X_{\zeta_1})(\omega)$ with $j\neq i_0$, define $\wt\La_{\zeta_1}(\omega)=j$; otherwise, define $\wt \La_{\zeta_1}(\omega)=i_0$.  Hence, $(\wt X_t,\wt \La_t)$ is well defined in the time interval $[0,\zeta_1]$. Besides, notice that under the condition (A1), the solution to SDE \eqref{b-5} is pathwise  unique, which means that $(\wt X_t,\wt \La_t)$ defined above is also pathwise unique in $[0,\zeta_1]$.

To proceed, denoting $\wt \La_{\zeta_1}(\omega_2)=i_1$, we define
\[\wt X_t(\omega_1,\omega_2)=X_t^{i_1}(\omega_1,\omega_2)\ \ \text{for $t\in (\zeta_1(\omega_2),\zeta_2(\omega_2)]$ and $\wt \La_t(\omega_1,\omega_2)=i_1$ for $t\in (\zeta_1(\omega_2),\zeta_2(\omega_2))$},\]
where $(X_t^{i_1})$ is the solution to \eqref{b-5} with initial value $X_{\zeta_1}^{i_1}(\omega_1,\omega_2)=\wt X_{\zeta_1}(\omega_1,\omega_2)$.
If $p(\zeta_2)(\omega)\in \Gamma_{i_1j}(\wt X_{\zeta_2})(\omega)$ with $j\neq i_1$, define $\wt \La_{\zeta_2}(\omega)=j$; otherwise, define $\wt \La_{\zeta_2}(\omega)=i_1$. Consequently, $(\wt X_t,\wt \La_t)$ is uniquely defined in $[\zeta_1,\zeta_2]$. Repeating this procedure, we can define the solution $(\wt X_t,\wt \La_t)$ for all $t\in [0,\infty)$ since $\lim_{n\to \infty}\zeta_n=\infty$. The arbitrariness of the Brownian motion $(B_t)_{t\geq 0}$ and the Poisson random measure $\mathcal N(\d t,\d z)$ implies that $(\wt X_t,\wt \La_t)_{t\geq 0}$ is a strong solution. The wellposedness of SDEs \eqref{b-1} and \eqref{b-2} has been established till now.

Next, we go to show that $(\wt X_t,\wt \La_t)_{t\geq 0}$ also satisfies  \eqref{a-2}.
Take a point $j_0\in\S$, and let  $f(i)=\mathbf1_{\{j_0\}}(i)$. Applying It\^o's formula for jump-diffusion processes, we get
  \begin{equation*}
  \begin{aligned}
    f(\wt \La_{t+\delta})&=f(\wt \La_t)\!+\!\int_{t}^{t+\delta}\!\! \int_{[0,\kappa_0]} f(\wt\La_{s-}+\vartheta(\wt X_s,\wt \La_{s-}, z))-f(\wt \La_{s-})\mathcal{N}(\d s, \d z)\\
    &=f(\wt \La_t)\!+\!  \int_t^{t+\delta}\!\sum_{j\in \S}q_{\wt\La_sj}(\wt X_s) ( f(j)-f(\wt\La_{s})) \d s \\
    &\qquad +\int_{t}^{t+\delta}\! \!\int_{[0,\kappa_0]} \!\! \big\{f(\wt\La_{s-}+\vartheta(\wt X_s,\wt \La_{s-}, z))-f(\wt \La_{s-})\big\}\wt{\mathcal{N}}(\d s, \d z),
    \end{aligned}
  \end{equation*}
  where $\wt{\mathcal{N}}(\d s,\d z)=\mathcal{N}(\d s,\d z)-\d s\d z$ denotes the compensator of the Poisson random measure $\mathcal{N}(\d s,\d z)$. Let $\F_t=\sigma((X_s,\La_s);s\leq t)$.  Taking the conditional expectation $\E[\,\cdot\,|\F_t]$ on both sides of the previous equation and using the Markov property of $(\wt X_t,\wt \La_t)_{t\geq 0}$, we obtain that
  \begin{equation*}
    \E\big[ f(\wt \La_{t+\delta})\big|\wt \La_t=i, \wt X_t=x\big] = f(i) +  \int_t^{t+\delta}\!\E\Big[\sum_{j\in \S}q_{\wt\La_sj}(\wt X_s) f(j)\Big|\wt \La_t=i, \wt X_t=x\Big]\d s.
  \end{equation*}
   This yields from $f(i)=\delta_{\{j_0\}}(i)$ that
  \begin{align*}
    \p\big(\wt \La_{t+\delta}=j_0|\wt \La_{t}=i,\wt X_t=x \big)&=\delta_{\{j_0\}}(i)+\int_t^{t+\delta} \!  \E \Big[q_{\wt \La_s j_0}(\wt X_s)  \Big|\wt \La_t=i, \wt X_t=x\Big]\d s.
  \end{align*}
  Then,
  \begin{equation}\label{b-4}
  \begin{aligned}
    &\frac 1\delta\Big(\p\big(\wt \La_{t+\delta}=j_0|\wt \La_t=i ,\wt X_t=x \big)-\delta_{\{j_0\}}(i)\Big)-q_{ij_0} (x)\\
    &=\frac1\delta \Big(\int_t^{t+\delta}\!\E\Big[\big(  q_{\wt \La_sj_0}(\wt X_s) -q_{ij_0}(x)\big)\Big|\wt \La_t=i,\wt X_t=x\Big]\d s\Big).
  \end{aligned}
  \end{equation}
%
Under the condition that for all $t\geq 0$, $\p(|\wt X_t|=\alpha_k)=0$,  $k=0,1,\ldots,m$, we obtain that
  \begin{align*}
  &\frac1\delta \Big(\int_t^{t+\delta}\!\E\Big[\big(  q_{\wt \La_sj_0}(\wt X_s) -q_{ij_0}(x)\big)\Big|\wt \La_t=i,\wt X_t=x\Big]\d s\Big)\\
    &=\frac 1\delta \int_{t}^{t+\delta}\!\E\Big[\sum_{k=1}^{m+1}\!\Big(q_{\wt \La_sj_0}^{(k)}\mathbf1_{[\alpha_{k-1}, \alpha_k)}(|\wt X_s|)-q_{ij_0}^{(k)}\mathbf1_{[\alpha_{k-1},\alpha_k)}(|x|)\Big)\Big|\wt \La_t=i,\wt X_t=x \Big]\d s\\
    &=\frac1\delta\sum_{k=1}^{m+1}\!\int_{t}^{t+\delta}\!\E\Big[ \big(q_{\wt\La_sj_0}^{(k)}-q_{ij_0}^{(k)}\big)\mathbf1_{(\alpha_{k-1},\alpha_k)}(|\wt X_s|)\\
    &\qquad \qquad \qquad+q_{ij_0}^{(k)}\big(\mathbf1_{(\alpha_{k-1},\alpha_k)}(|\wt X_s|)-\mathbf1_{(\alpha_{k-1},\alpha_k)}(|x|)\big)\Big|\wt\La_t=i,\wt X_t=x\Big]\d s\\
    &\longrightarrow 0,\qquad \qquad \text{as}\ \delta\downarrow 0,
  \end{align*}
  by using the dominated convergence theorem and  the facts \[\lim_{s\downarrow t}\wt \La_s=\wt \La_t,\ \ \text{a.s. and } \lim_{s\downarrow t}\mathbf1_{(\alpha_{k-1},\alpha_k)}(|\wt X_s|)  =\mathbf1_{ (\alpha_{k-1},\alpha_k)}(|x|), \ \ \text{a.s.}.\]
  Hence,    \eqref{a-2} holds and we complete the proof.
\end{proof}

\begin{myrem}
1) This theorem can be easily modified to prove the wellposeness of the hybrid system in $\R$ satisfying \eqref{a-1}, \eqref{a-2} with $(q_{ij}(x))$ in the form \eqref{a-4}.

2) If $a(x,i)=\sigma(x,i)\sigma^\ast(x,i)$ is nondegenerate, so the density of the distribution of $\wt X_t$ exists, and hence $\p(|\wt X_t|=\alpha_k)=0$, $t>0$. Furthermore, see, for example, Skorokhod \cite[Chapter I,Section 2]{Sko} for the existence of transition densities for nondegenerate and degenerate diffusion processes.
\end{myrem}

\section{Approximation limit of hybrid systems with piecewise constant type switching}

In this section, we shall consider the approximation problem of a stochastic hybrid system via a sequence of stochastic hybrid systems with piecewise constant type switching. Precisely, a stochastic hybrid system $(X_t,\La_t)_{t\geq 0}$ is given, which satisfies \eqref{a-1} and \eqref{a-2} with $x\mapsto q_{ij}(x)$ being continuous for every $i,j\in\S$.  Consider an approximation sequence $\{q_{ij}^{(n)}(x)\}_{n\geq 1} $ in the form \eqref{a-3} or \eqref{a-4} when $d=1$ to the continuous function  $q_{ij}(x)$.  Associated with $(q_{ij}^{(n)}(x))_{i,j\in\S}$, there is a sequence of hybrid systems $(X_t^{(n)},\La_t^{(n)})_{t\geq 0}$ due to Theorem \ref{thm-1} under suitable conditions. The purpose of this section is to show the convergence of the distributions of $X_t^{(n)}$ to that of $X_t$ in the $L_1$-Wasserstein distance as $n\to\infty$, and the convergence rate is estimated in terms of the difference between $(q_{ij}^{(n)}(x))_{i,j\in\S}$ and $(q_{ij}(x))_{i,j\in\S}$.



Let $Q(x)=(q_{ij}(x))_{i,j\in\S}$ be a conservative, irreducible $Q$-matrix on $\S$ for every $x\in \R^d$. Assume that $x\mapsto q_{ij}(x)$ is continuous and $\tilde\kappa_0:=\sup_{x\in \R^d}\max_{i\in\S} q_i(x)<\infty$, where $q_i(x)=-q_{ii}(x)$.
Suppose that $Q^{(n)}(x)=(q_{ij}^{(n)}(x))_{i,j\in\S}$ is a sequence of $Q$-matrices of piecewise constant type:
\begin{equation}\label{d-1}
q_{ij}^{(n)}(x)=\sum_{k=1}^{m_n+1}\!q_{ij}^{n,k}\mathbf1_{[\alpha_{k-1}^n,\alpha_k^n)}(|x|)
\end{equation}
where $\Delta_{m_n}^n:=\{0=\alpha_{0}^n<\alpha_1^n< \ldots<\alpha_{m_n}^n<\alpha^n_{m_n+1}=+\infty\}$ is a finite partition of $[0, \infty)$, and $(q_{ij}^{n,k})_{i,j\in\S}$ is a conservative, irreducible $Q$-matrix on $\S$ for every $n\geq 1$, $k=1,\ldots,m_n$.
Assume
\begin{equation}\label{d-2}
\Theta_n:=\! \sup_{x\in \R^d}\!\|Q^{(n)}(x)\!-\!Q(x)\|_{\ell_1}\!=\!\sup_{x\in \R^d}\max_{i \in\S} \sum_{j\neq i} \big|q_{ij}^{(n)}(x)\!-\!q_{ij}(x)\big|\longrightarrow 0,\quad \text{as $n\to \infty$}.
\end{equation}
Let $\kappa_1$ be a constant such that
\begin{equation} \label{d-2.5}
\kappa_1\geq \max\big\{\tilde \kappa_0, |q_{ii}^{n,k}|;\, n\geq 1,1\leq k\leq m_n+1\big\}.
\end{equation}

According to Skorokhod's representation theorem, Theorem \ref{thm-1}, our concerned  stochastic hybrid system $(X_t,\La_t)_{t\geq 0}$ can be expressed as a solution to the SDE
\begin{equation}\label{d-3}
\left\{ \begin{array}{l}
\d X_t=b(X_t,\La_t)\d t+\sigma(X_t,\La_t)\d B_t,\\
\d \La_t=\int_{[0,\kappa_1]} \!\vartheta(X_t,\La_{t-},z)\mathcal{N}(\d t,\d z)
\end{array}
\right.
\end{equation}
with initial value $(X_0,\La_0)=(x_0,i_0)\in \R^d\!\times\!\S$,
where $(B_t)_{t\geq 0}$ is a $d$-dimensional Brownian motion; $\mathcal{N}(\d t,\d z)$ is a Poisson random measure with intensity $\d t\!\times\!\d z$ supported on $[0,\infty)\times [0,\kappa_1]$; $(B_t)$ and $\mathcal{N}(\d t,\d z)$ are mutually independent; $\vartheta(x,i,z)$ is defined as in \eqref{g-2} associated with $(q_{ij}(x))_{i,j\in\S}$ given above.

Corresponding to $Q^{(n)}(x)$ given in \eqref{d-1} satisfying the approximation condition \eqref{d-2}, we consider the approximation processes
$(X_t^{(n)},\La_t^{(n)})_{t\geq 0}$ given as the solutions to SDEs
\begin{equation}\label{d-4}
  \left\{ \begin{array}{l} \d X_t^{(n)}=b(X_t^{(n)},\La_t^{(n)})\d t+\sigma(X_t^{(n)},\La_t^{(n)})\d B_t,\\
  \d \La_t^{(n)}=\int_{[0,\kappa_1]}\vartheta^{(n)}(X_t^{(n)},\La_{t-}^{(n)},z)\mathcal{N}(\d t,\d z)
  \end{array}
  \right.
\end{equation} with initial value $(X_0^{(n)},\La_0^{(n)})=(x_0,i_0)\in \R^d\!\times\!\S$, where $\vartheta^{(n)}(x,i,z)$ is defined as in \eqref{g-2} with intervals $\{\Gamma_{ij}^{(n)}(x)\}$ being associated with $(q_{ij}^{(n)}(x))_{i,j\in\S}$. Theorem \ref{thm-1} ensures the existence of $(X_t^{(n)},\La_t^{(n)})$ under the conditions (A1) and (A2).

Before presenting our main result, let us make some preparations, which play crucial role in the argument. We   use the idea of \cite{SZ21} to estimate the difference between $X_t$ and $X_t^{(n)}$, and we improve the convergence rate based on the current construction of
 $\Gamma_{ij}(x)$ given in \eqref{g-1}.


\begin{mylem}\label{lem-4.1}
For any two Borel sets $A$, $\Gamma$ in $\R$, denote $A\Delta \Gamma=(A\backslash \Gamma)\cup (\Gamma\backslash A)$ and $|A\Delta \Gamma|$ the Lebesgue measure of $A\Delta \Gamma$.   Then, for any $i,j\in\S$,
\[\big|\Gamma_{ij}(x)\Delta   \Gamma_{ij}^{(n)}(y)\big|\leq \max_{i,j\in\S}\big|q_{ij}(x)- q_{ij}^{(n)}(y)\big|,\quad
x,y\in \R^d.
\]
\end{mylem}
\begin{proof}
  This estimate follows immediately from the construction method of $\Gamma_{ij}(x)$ and $\Gamma_{ij}^{(n)}(x)$  as in  \eqref{g-1}.
\end{proof}

\begin{mylem}\label{lem-4.2}
It holds that
\begin{equation}\label{e-5}
\frac 1t \int_0^t\!\p(\La_s\neq \La^{(n)}_s)\d s\leq \int_0^t\!\E\big[\|Q(X_s)-Q^{(n)}\!(X_s^{(n)})\|_{\ell_1}\! \big]\d s,\qquad t>0.
\end{equation}
\end{mylem}
\begin{proof}
  Let $(p(t))_{t\ge 0}$ be the Poisson point process associated with $\mathcal{N}(\d t,\d z)$. Set $\zeta_1<\zeta_2<\cdots<\zeta_n<\cdots$ be the sequence of jumping time of $(p(t))_{t\geq 0}$. Put $\zeta_0=0$, $\tau_k=\zeta_{k}-\zeta_{k-1}$ for $k\geq 1$, and
\[N(t)=\#\{k\geq 1;\ \zeta_k\leq t\}, \ \text{the number of jumps of $(p(t))_{t\geq 0}$ before time $t$.}\]
Denote by $\Delta p(t)=p(t)-p(t-)$. So, $\Delta p(\zeta_k)>0$ for $k\geq 1$.
According to \cite[Chapter 1]{Ike},  $\Delta p(\zeta_k)$ is independent of $\zeta_m$ for $k,m\in \N$, and
\begin{align*}
  &\p(\Delta p(\zeta_k)\in \d x)=\frac1{\kappa_1} \d x,\quad \p(\tau_k>t)=\e^{-\kappa_1 t}, \\
  &\p(N(t)=k)=\frac{(\kappa_1 t)^k}{k!}\e^{-\kappa_1 t},\quad t>0, \, k\geq 1,
\end{align*} where $\kappa_1$ is the same as in \eqref{d-2.5}.

Letting $\delta\in(0,1)$, $t\in (0,\delta)$, due to \eqref{d-3}, \eqref{d-4}, we have
\begin{align*}
  \p(\La_t\neq \La^{(n)}_t)&= \p(\La_t\neq \La^{(n)}_t, N(t)\geq 1)\\
  &=\p(\La_t\neq \La^{(n)}_t, N(t)=1)+\p(\La_t\neq \La^{(n)}_t, N(t)\geq 2)\\
  &\leq \p(\La_t\neq \La^{(n)}_t,N(t)=1)+\kappa_1^2\delta^2.
\end{align*}
Furthermore,
\begin{equation}\label{e-9}
\begin{aligned}
  \p(\La_t\neq \La^{(n)}_t,N(t)=1)&=\int_0^t \p(\La_s\neq \La^{(n)}_s,\zeta_1\in \d s,\zeta_2>t)\\
  &=\int_0^t\p\Big(\Delta p(s)\in \bigcup_{j\neq i_0}\big\{\Gamma_{i_0j}(X_s)\Delta \Gamma_{i_0j}^{(n)}(X^{(n)}_s)\big\}, \zeta_1\in \d s,\zeta_2>t\Big)\\
  &=\int_0^t\p\Big(\Delta p(s)\in \!\bigcup_{j\neq i_0}\big\{\Gamma_{i_0j}(X_s^{i_0})\Delta \Gamma_{i_0j}^{(n)}(X_s^{(n),i_0})\big\}, \zeta_1\!\in\!\d s,\zeta_2>t\Big),
  \end{aligned}
  \end{equation}
  where $(X_t^{i_0})$ and $(X_t^{(n),i_0})$ denote respectively the solution to SDEs \eqref{d-3}, \eqref{d-4} by replacing $\La_t$ and $\La_t^{(n)}$ in the coefficients $b$ and $\sigma$ with $i_0$. By Lemma \ref{lem-4.1}, it follows from the independence between $(p(t))$ and $(B_t)$ that
  \begin{equation}\label{e-10}
  \begin{aligned}
  \p(\La_t\neq \La^{(n)}_t,N(t)=1)&\leq \frac{1}{\kappa_1}\int_0^t  \E\big[\|Q(X_s)-Q^{(n)}(X_s^{(n)})\|_{\ell_1} \big]\p(\zeta_1\in \d s,\zeta_2>t)\\
  &\leq   \int_0^t\!\E\big[\|Q(X_s)- Q^{(n)}(X_s^{(n)})\|_{\ell_1}\big] \d s.
\end{aligned}
\end{equation}
Hence,
\begin{equation}\label{e-6}
\p(\La_t\neq \La^{(n)}_t)\leq \kappa_1^2\delta^2+\int_0^t\!\E\big[\|Q(X_s)- Q^{(n)}(X_s^{(n)})\|_{\ell_1}\big] \d s,\quad t\in (0,\delta].
\end{equation}
Next, let us consider the case $\p(\La_{2\delta}\neq \La_{2\delta}^{(n)})$.
\begin{align*}
  \p(\La_{2\delta}\neq \La^{(n)}_{2\delta})&=\p(\La_{2\delta}\neq \La^{(n)}_{2\delta},\La_\delta=\La^{(n)}_{\delta}) +\p (\La_{2\delta}\neq \La^{(n)}_{2\delta},\La_\delta\neq \La^{(n)}_\delta)\\
  &\leq \p(\La_{2\delta}\neq \La^{(n)}_{2\delta}\big|\La_{\delta}=\La^{(n)}_\delta)+\p (\La_\delta\neq \La^{(n)}_\delta)\\
  &\leq \p(\La_{2\delta}\neq \La^{(n)}_{2\delta}, N(2\delta)\!-\!N(\delta)\geq 1\big| \La_\delta=\La^{(n)}_\delta)+\p(\La_\delta\neq \La^{(n)}_\delta).
\end{align*}
Similar to the deduction of \eqref{e-6}, using still the independent increment of the processes $(p(t))$ and $(B_t)$, we obtain that
\begin{align*}
  \p(\La_{2\delta}\neq \! \La^{(n)}_{2\delta},N(2\delta)\!-\!N(\delta)\! \geq 1\big|\La_\delta\! =\!\La^{(n)}_\delta)\leq \kappa_1^2\delta^2\! + \! \int_\delta^{2\delta}\! \E\big[\|Q(X_s)\!-\!Q^{(n)}(X_s^{(n)})\|_{\ell_1}\big] \d s.
\end{align*}
Combining this with \eqref{e-6}, we get
\begin{equation}\label{e-7}
\p(\La_{2\delta}\neq\! \La^{(n)}_{2\delta})\leq 2\kappa_1^2\delta^2\!+\! \int_0^{2\delta}\!\E \big[\|Q(X_s)-Q^{(n)}(X_s^{(n)})\|_{\ell_1} \big] \d s.
\end{equation}
Deducing recursively, it holds
\begin{equation}\label{e-8}
\p(\La_{k\delta}\neq \La^{(n)}_{k\delta})\leq k\kappa_1^2\delta^2+\!\int_0^{k\delta}\!\E \big[\|Q(X_s)-Q^{(n)}(X_s^{(n)})\|_{\ell_1}\big] \d s,\quad \ k\geq 1.
\end{equation}
Given   $t>0$, let $K=[t/\delta]$, $t_k=k\delta$ for $k=0,1,\ldots, K$, and $t_{K+1}=t$. Then
\begin{align*}
  \int_0^t\!\p(\La_s\!\neq \!\La^{(n)}_s)\d s&=\sum_{k=0}^K\int_{t_k}^{t_{k+1}}\!\! \big(\p (\La_s\neq \La^{(n)}_s, \La_{t_k}=\La^{(n)}_{t_k})  \!+\! \p(\La_s\neq \La^{(n)}_s, \La_{t_k}\neq \La^{(n)}_{t_k})\big)\d s \\
  &\leq\sum_{k=0}^{K}\!\int_{t_k}^{t_{k+1}}\!\! \p(\La_s\neq \La^{(n)}_s|\La_{t_k}=\La^{(n)}_{t_k})\d s+\sum_{k=0}^K\delta\,\p(\La_{t_k}\neq \La^{(n)}_{t_k})\\
  &\leq \sum_{k=0}^K\delta\,\p(\La_{t_k}\!\neq \! \La^{(n)}_{t_k})\!+\! \sum_{k=0}^K\!\int_{t_k}^{t_{k+1}}\!\! \! \Big(\!\int_0^\delta\! \! \E\big[\|Q(X_{t_k+r})\!-\!Q^{(n)}(  X^{(n)}_{t_k+r})\|_{\ell_1}\big]\d r\!+\!\kappa_1^2\delta^2\Big)\d s\\
  &\leq \sum_{k=0}^K\!\delta\Big(k\kappa_1^2\delta^2\!+\! \! \int_0^{t_k}\!\! \E\big[\|Q(X_s)-Q^{(n)}(X_s^{(n)})\|_{\ell_1} \!\big] \d s \Big)\\ &\qquad + \sum_{k=0}^K\!\int_{t_k}^{t_{k+1}}\!\! \! \int_0^\delta\! \! \E\big[\|Q(X_{t_k+r})\!-\!Q^{(n)}(  X^{(n)}_{t_k+r})\|_{\ell_1}\big]\d r\d s +\kappa_1^2\delta^2 t.
\end{align*}
Letting $\delta\downarrow 0$, this yields that
\begin{equation*}
  \int_0^t\!\p(\La_s\neq \La^{(n)}_s)\d s\leq t\int_0^{t}\! \E\big[\|Q(X_s)-Q^{(n)}(X_s^{(n)})\|_{\ell_1} \!\big] \d s,
\end{equation*} which is the desired conclusion.
\end{proof}

For any two probability measures $\mu$ and $\nu$ on $\R^d$,  the $L_1$-Wasserstein distance between $\mu$ and $\nu$ is defined by
    \[\W_1(\mu,\nu)=\inf_{\pi\in \C(\mu,\nu)}\Big\{\int_{\R^d\!\times \R^d}\!\! |x-y|\pi(\d x,\d y)\Big\},\]
    where $\C(\mu,\nu)$ stands for the set of all couplings of $\mu$ and $\nu$.

\begin{mythm}\label{thm-4.1}
Consider the hybrid systems $(X_t,\La_t)$ and $(X_t^{(n)},\La_t^{(n)})$ satisfying \eqref{d-3} and \eqref{d-4} respectively.
 Assume  $\sigma(x,i)=\!\sigma\!\in\!\R^{d\times d}$ with determinant $\mathrm{det}(\sigma)>0$, and   $b(x,i)=\hat{b}(x,i)+Z(x)$ satisfies
  \begin{equation}\label{d-5}
  |\hat{b}(x,i)-\hat{b}(y,i)|+|Z(x)-Z(y)|\leq K_2|x-y|,\qquad x,y\in \R^d,\,i\in\S,
  \end{equation}
  and
  \begin{equation}\label{d-6}
  \max_{i\in\S}\sup_{x\in \R^d} |\hat{b}(x,i)|\leq K_2 \quad \text{for some $K_2>0$.}
  \end{equation}  Assume that there exists $K_3>0$ such that
  \begin{equation}\label{d-7}
  |q_{ij}(x)-q_{ij}(y)|\leq K_3|x-y|,\quad x,y\in \R^d,\ i,j\in\S, \quad \text{and} \ \ \sup_{x\in\R^d} \max_{i\in \S} q_i(x)<\infty.
  \end{equation} Suppose \eqref{d-2} holds. Denote by $\mu_t^{n}$ the distribution of $X_t^{(n)}$ and $\mu_t$ the distribution of $X_t$.
    Then
    \begin{equation}\label{d-9}
    \sup_{t\in [0,T]}\W_1(\mu_t^n,\mu_t)\leq 2T \e^{(K_2+2(N\!-\!1)K_3)T} \Theta_n,\qquad T>0,
    \end{equation}
    where $\Theta_n$ is defined as in \eqref{d-2}. Moreover,
   $\lim_{n\to \infty} \sup_{t\in [0,T]}\W_1(\mu_t^n,\mu_t)=0$.
\end{mythm}

\begin{proof}
  Under the conditions \eqref{d-5}-\eqref{d-7}, for any initial value $(x,i)\in \R^d\times \S$ the solution $(X_t,\La_t)$ to
  SDE \eqref{d-3} exists uniquely in the pathwise sense. The same assertion holds for   $(X_t^{(n)},\La_t^{(n)})$ as a  solution to SDE \eqref{d-4}.
  By virtue of \eqref{d-3}, \eqref{d-4}, \eqref{d-5}, \eqref{d-6}, for $T>0$,
  \begin{align*}
    &\E\big[\sup_{0\leq t\leq T} |X_t^{(n)}-X_t|\big]\\
    &\leq \E\Big[\int_0^T\! |b(X_s^{(n)},\La_s^{(n)})-b(X_s,\La_s)|\d s\Big]\\
    &\leq \E\Big[\int_0^T\!\! |\hat{b}(X_s^{(n)},\La_s^{(n)}) \!-\!\hat{b}(X_s,\La_s^{(n)})| \! +\!|Z(X_s^{(n)}) \! -\!Z(X_s)|\!+\! |\hat{b}(X_s,\La_s^{(n)})\!-\! \hat{b}(X_s,\La_s)|\d s\Big]\\
    &\leq K_2\int_0^T\E\Big[|X_s^{(n)}-X_s|+2\mathbf1_{\{ \La_s^{(n)}\neq \La_s\}}\Big]\d s.
  \end{align*}
  According to  Lemma \ref{lem-4.2},
  \begin{equation}\label{d-8}
  \int_0^T\p(\La_s^{(n)}\neq \La_s)\d s\leq T\int_0^T\E\big[\|Q(X_s)\!-\! Q^{(n)}(X_s^{(n)})\|_{\ell_1}\big]\d s.
  \end{equation}
  Combining this with \eqref{d-7}, we get
  \begin{equation*}
  \begin{split}
  \int_0^T\p(\La_s^{(n)}\neq \La_s)\d s&\leq \!T\int_0^T\!\!\E\Big[\|Q(X_s)\!-\! Q(X_s^{(n)})\|_{\ell_1}+\|Q(X_s^{(n)})- Q^{(n)}(X_s^{(n)})\|_{\ell_1} \Big]\d s\\
  &\leq T\int_0^T\! \Big((N-1)K_3\E\big[|X_s^{(n)}-X_s|\big]\!+\! \Theta_n\Big)\d s.
  \end{split}
  \end{equation*}
  Hence,
  \begin{equation*}
  \begin{split}
    \E\Big[\sup_{0\leq t\leq T}|X_t^{(n)}\!-\!X_t|\Big] &\leq 2T^2\Theta_n\!+\! \big(K_2\!+\!2(N\!-\!1)TK_3\big) \int_0^T\E\Big[\sup_{0\leq s\leq t}|X_s^{(n)}\!-\!X_s|\Big] \d t,
  \end{split}
  \end{equation*}
  which implies that
  \[\E\Big[\sup_{0\leq t\leq T}|X_t^{(n)}-X_t|\Big]\leq 2T^2 \e^{(K_2+2(N\!-\!1)TK_3)T}\Theta_n,\]
  by Gronwall's inequality. The desired conclusion \eqref{d-9} follows immediately from the fact $\W_1(\mu_t^n,\mu_t)\leq \E|X_t^{(n)}-X_t|$, and the proof is complete.
\end{proof}

\begin{myrem}
  In Theorem \ref{thm-4.1}, the $L_1$-Wasserstein distance cannot be replaced by other $L_p$-Wasserstein distance with $p>1$, which is due to the application of Gronwall's inequality and the estimate on the quantity $\int_0^t\p(\La_s\neq \La_s^{(n)})\d s$ obtained in Lemma \ref{lem-4.2}.
\end{myrem}



%

\section{Ergodicity and stability  of hybrid systems with piecewise constant type switching}

The stability and ergodicity of regime-switching diffusion processes have been widely studied and applied in the literatures. Providing easily verifiable conditions on the stability and ergodicity of state-dependent regime-switching processes is a quite challenging task because of the intensive interaction between the continuous component $X_t$ and the discrete component $\La_t$.
However, as a special kind of   stochastic hybrid systems with state-dependent switching, we shall show that the ergodicity and stability of stochastic hybrid systems  with  piecewise constant type switching satisfying \eqref{a-3} or \eqref{a-4} when $d=1$ can be easily justified, whose justification criteria are very similar to those for the stochatic hybrid systems with state-independent switching.   Invoking the approximation limit studied in Section 3, we can widely apply the hybrid systems with piecewise constant type switching as feasible mathematical models in realistic practice.

\subsection{Criteria on  stability and instability in probability}
Recall the hybrid system $(X_t,\La_t)$ satisfying \eqref{a-1}, \eqref{a-2} and \eqref{a-3}, whose infinitesimal generator $\mathscr{A}$ is given by
\begin{align}\label{f-0}
  \mathscr{A} f(x,i)&=\mathcal{L}^{(i)} f(\cdot,i)(x)+Q(x)f(x,\cdot)(i)\\ \notag
  &:=\sum_{k=1}^d b_k(x,i)\frac{\partial f}{\partial x_k}(x,i)\!+\!\frac 12 \sum_{k,l=1}^d\!a_{kl}(x,i)\frac{\partial^2 f}{\partial x_k\partial x_l}(x,i)\!+ \! \sum_{j\neq i} q_{ij}(x)\big(f(x,j)\!-\!f(x,i)\big)
\end{align}
for any $f\in C_b^2(\R^d\!\times\!\S)$. Here $a(x,i)=\sigma(x,i)\sigma^\ast(x,i)$, and $\mathcal{L}^{(i)}$ is the generator associated with the diffusion process $(X_t^{(i)})$ in the fixed environment $i\in\S$, which satisfies the following SDE:
\[\d X_t^{(i)}=b(X_t^{(i)},i)\d t+\sigma(X_t^{(i)},i)\d B_t.\]

In this section, we are interested in   the stability of the equilibrium point $x=0$. Adopting Khasminskii's notations, the equilibrium point $x=0$ is said to be \emph{stable in probability} if for any $\veps>0$, $\lim_{x\to 0}\p\big(\sup_{t\geq 0} |X_t^{x,i}|>\veps\big)=0$ for every $i\in \S$; is said to be \emph{unstable in probability} if it is not stable in probability. $x=0$ is said to be \emph{asymptotically stable in probability} if it is stable in probability and further
$\lim_{x\to 0}\p\big(\lim_{t\to\infty} X_t^{x,i}=0\big)=1$ for every $i\in \S$.

Assume the coefficients satisfy the following conditions in this subsection.
\begin{itemize}
  \item[$(\mathrm{H1})$] $b(0,i)=0$, $\sigma(0,i)=0$ for every $i\in\S$. Moreover, for any sufficiently small $r_0>\veps>0$, there exist $l\in \{1,2,\ldots,d\}$ and $c(\veps)>0$ such that $a_{ll}(x,i)>c(\veps)$ for all $(x,i)\in \{x;\,\veps<|x|<r_0\}\!\times\!\S$.
\end{itemize}

As a Markovian process, the Foster-Lyapunov condition to justify the stability of $x=0$ still works for the regime-switching process $(X_t,\La_t)$; cf. e.g. \cite[Chapter 7]{YZ}. Moreover, much effort has been done in \cite{SW18,SX,YZ} to provide more explicit conditions to guarantee the existence of desired Lyapunov functions. For the convenience of readers, we summarize some known results as follows.
\begin{mylem}\label{lem-3}
Assume $\mathrm{(H1)}$ and $\mathrm{(A1)}$, $\mathrm{(A2)}$ hold. Then the following assertions hold.
\begin{itemize}
  \item[$\mathrm{(i)}$] If $d=1$, then $\p(X_t^{x,i}\neq 0,\ \forall\,t>0)=1$ for any $x\neq 0$, $i\in \S$. Here we denote $(X_t^{x,i},\La_t^{x,i})$ the solution to \eqref{a-1}, \eqref{a-2} with initial value $(X_0^{x,i},\La_0^{x,i})=(x,i)$.
  \item[$\mathrm{(ii)}$] If there exists a neighborhood $D$ of $0$ in $\R^d$ and a nonnegative function $V:D\!\times\!\S\to [0,\infty)$ such that $V(\cdot,i)$ is continuous in $D$ and $\inf_{i\in\S} V(x,i)$ vanishes only at $x=0$; $V\in C^2(D\backslash\{0\}\times\S)$; $\mathscr{A}V(x,i)\leq 0$ for all $x\in D\backslash\{0\}$, $i\in\S$. Then the equilibrium point $x=0$ is asymptotically stable in probability.
  \item[$\mathrm{(iii)}$] If there exists a neighborhood $D$ of $0$ in $\R^d$ and a function $V:D\times\S\to [0,\infty)$ such that $V\in C^2(D\backslash\{0\}\times \S)$, $\mathscr{A}V(x,i)\leq 0$ for any $x\in D\backslash\{0\}$, $i\in\S$, and $\lim_{x\to 0} V(x,i)=\infty$. Then the equilibrium point $x=0$ is unstable in probability.
\end{itemize}
\end{mylem}
Notice that in the proof of Lemma \ref{lem-3}(i) in \cite[Lemma 7.1]{YZ}, the assumption that $x\mapsto q_{ij}(x)$ is continuous has not been used. Hence, this assertion still works for hybrid systems with piecewise constant type switching.

To proceed, let us introduce two types of conditions in order to deal with separately linear systems or nonlinear systems, which have been explored in \cite{SX}.
\begin{itemize}
  \item[$\mathrm{(L1)}$]  There exist constants $\beta_i\in\R$ for every $i\in\S$, a neighborhood $D$ of $0$ in $\R^d$, a function $\rho:D\backslash\{0\}\to (0,\infty)$ satisfying $\rho\in C^2(D\backslash\{0\})$ such that
      \[\mathcal{L}^{(i)}\rho(x)\leq \beta_i\rho(x),\quad \ \forall\,x\in D\backslash\{0\},\, i\in\S.
      \]
  \item[$\mathrm{(L2)}$] There exist constants $\beta_i\in\R$ for every $i\in\S$, a neighborhood $D$ of $0$ in $\R^d$, functions $\rho,h:D\backslash\{0\}\to (0,\infty)$ satisfying $\rho,\,h\in C^2(D\backslash\{0\})$ such that
      \[\mathcal{L}^{(i)} \rho(x)\leq \beta_ih(x),\quad \forall\, x\in D\backslash\{0\},\,i\in\S,\ \text{and}\  \lim_{x\to 0}\frac{h(x)}{\rho(x)}=0,\ \lim_{x\to 0}\frac{\mathcal{L}^{(i)}h(x)}{h(x)}=0.\]
\end{itemize}
Based on (L1) and (L2), \cite{SX} established some sufficient conditions on the stability and unstablity of Markovian regime-switching. These conditions are quite sharp, which can provide cut-off type criteria for some concrete examples including linear hybrid systems and non-linear hybrid systems. In next theorem, we shall see that the $Q$-matrix $(q_{ij}^{(1)})$ in the piecewise constant type switching $(q_{ij}(x))$  given in \eqref{a-3} will play the dominant role in studying the stability of $(X_t,\La_t)$ at the equilibrium point $x=0$. In the following, $\bm{\xi}=(\xi_1,\ldots,\xi_N)\gg 0$ means that $\xi_i>0$ for all $i=1,\ldots,N$.

\begin{mythm}\label{thm-3}
Let $(X_t,\La_t)$ be a hybrid system satisfying \eqref{a-1}, \eqref{a-2}, \eqref{a-3}. Assume $\mathrm{(H1)}$, $\mathrm{(A1)}$, $\mathrm{(A2)}$ hold.  Let $\pi^{(1)}=(\pi_i^{(1)})_{i\in\S}$ be the invariant probability measure of $(q_{ij}^{(1)})$. Suppose one of $\mathrm{(L1)}$ and  $\mathrm{(L2)}$ holds with $\sum_{i\in\S}\pi_i^{(1)}\beta_i<0$.  Then the equilibrium point $x=0$ is asymptotically stable in probability if $\rho(x)$ vanishes only at $0$, and is unstable in probability if $\lim_{x\to 0} \rho(x)=\infty$.
\end{mythm}

\begin{proof}
  Let $\wt D=D\cap\{x;|x|<\alpha_1\}$ with $D$ in (L1) or (L2), which is a new neighborhood of $0$ in $\R^d$.

\noindent \textbf{Case 1}. Suppose (L1) holds.
According to the Perron-Frobenius theorem, \[-\eta_p:=\max\big\{\mathrm{Re}\gamma;\,\gamma \ \text{in the spectrum of $Q^{(1)}+p\,\mathrm{diag}({\bm \beta})$}\big\} \] is a simple eigenvalue  of $Q^{(1)}+p\,\mathrm{diag}({\bm \beta})$  with an eigenvector ${\bm\xi}=(\xi_i)_{i\in\S}\gg 0$, where $\mathrm{diag}({\bm \beta})$ denotes the diagonal matrix generated by the vector ${\bm \beta}=(\beta_1,\ldots,\beta_N)$.  By virtue of \cite[Proposition 4.2]{Bar}, when $\sum_{i\in\S} \pi_i^{(1)}\beta_i<0$, there exists  $p_0>0$ such that for any $0<p<p_0$, $\eta_p>0$,
\begin{equation}\label{f-1} \big(Q^{(1)}+p\,\mathrm{diag}({\bm \beta})\big){\bm \xi}(i)=-\eta_p\xi_i<0,\qquad i\in\S.
\end{equation}
For $p\in (0,p_0)$, let $V(x,i)=\xi_i\rho(x)^p$, then
\[\mathscr{A}V(x,i)=(Q^{(1)}{\bm\xi}(i) +p\beta_i\xi_i)\rho(x)^p =-\eta_p\xi_i\rho(x)^p<0,\quad x\in \wt D,\ i\in \S.\]
Since $(\xi_i)$ is positive and bounded, by virtue of Lemma \ref{lem-3}(ii),(iii), we obtain that $x=0$ is asymptotically stable in probability if $\rho(x)$ vanishes only at $0$, and is unstable in probability if $\lim_{x\to 0}\rho(x)=\infty$.

\noindent \textbf{Case 2}. Suppose (L2) holds. According to the Fredholm alternative, it follows from the fact $\sum_{i\in\S}\pi_i^{(1)}\beta_i<0$ that there exist  a $c>0$ and a vector ${\bm\xi}=(\xi_i)_{i\in\S}\gg 0$ such that
\begin{align*}
  Q(x){\bm \xi}(i)&=\sum_{j\neq i}q_{ij}(x)(\xi_j-\xi_i)=\sum_{j\neq i}q_{ij}^{(1)}(\xi_j-\xi_i)\\
  &=-c-\beta_i,\qquad i\in \S, \ x\in \wt{D}.
\end{align*}
Put $V(x,i)=\rho(x)+\xi_i h(x)$, then
\begin{align*}
  \mathscr{A}V(x,i)&=\mathcal{L}^{(i)}\rho(x)+Q(x){ \bm \xi}(i)h(x)+\xi_i\mathcal{L}^{(i)}h(x)\\
  &\leq \Big(\beta_i+Q^{(1)}{\bm \xi}(i)+\xi_i \frac{\mathcal{L}^{(i)}h(x)}{h(x)}\Big) h(x)\\
  &\leq \Big(-c+\xi_i \frac{\mathcal{L}^{(i)}h(x)} {h(x)}\Big)h(x),\quad x\in \wt D, \ i\in\S.
\end{align*}
Since ${\bm\xi}$ is bounded and $\lim_{x\to 0}\frac{\mathcal{L}^{(i)}h(x)}{h(x)}=0$, there exists $\delta>0$ such that $\forall\, x\in\R^d$ with $0<|x|<\delta$ it holds
\[\frac{\mathcal{L}^{(i)}h(x)}{h(x)}\leq \frac{c}{2}.\]
Thus, we get
\[\mathscr{A}V(x,i)\leq -\frac{c}2 h(x)\leq 0,\quad \ \forall \,x\in \wt{D}\cap\{x;|x|<\delta\}, \,x\neq 0,\ i\in\S,\] which implies the desired conclusion due to Lemma \ref{lem-3}(ii),(iii).
\end{proof}

\begin{myrem}
  For the case that the assumption (L1) holds, Theorem \ref{thm-3} can also be proved using the results in \cite{NY18}.
\end{myrem}

Beyond the piecewise constant type switching $(q_{ij}(x))$ studied in Theorem \ref{thm-3}, we go to consider another complicated  case in $\R$:
\begin{equation}
\label{f-2}
q_{ij}(x)=q_{ij}^{(1)} \mathbf1_{ [0, \alpha_1) }(x)\!+\!\tilde q_{ij}^{(1)} \mathbf1_{(\alpha_{-1}, 0)}(x)\!+\!\sum_{k=2}^{m_1}q_{ij}^{(k)} \mathbf1_{[\alpha_{k-1},  \alpha_k)}(x)\!+\!\sum_{l=1}^{m_2} \tilde q_{ij}^{(l)}
\mathbf1_{(\alpha_{-l-1}, \alpha_{-l}]}(x),
\end{equation} where $\{-\infty=\alpha_{-m_2-1}<\ldots<\alpha_{-1}<0 <\alpha_1<\ldots<\alpha_{m_1}=\infty\}$ is a finite partition of $\R$, and $(q_{ij}^{(k)})$, $(\tilde q_{ij}^{(l)})$ are conservative, irreducible $Q$-matrices for $k\in \{1,\ldots,m_1\}$ and $l\in\{1,\ldots,m_2\}$. Notice that in this situation, there are two different $Q$-matrices $(q_{ij}^{(1)})$ and $(\tilde q_{ij}^{(1)})$ in any neighborhood of the equilibrium point $x=0$. The corresponding Markovian regime-switching processes associated respectively with the transition rate matrix $(q_{ij}^{(1)})$ and $(\tilde q_{ij}^{(1)})$ may own quite different stability at the equilibrium point $x=0$.

\begin{mythm}\label{thm-4}
Let $(X_t,\La_t)$ be a hybrid system satisfying \eqref{a-1}, \eqref{a-2} with $d=1$ and $(q_{ij}(x))$ being given in \eqref{f-2}. Assume $\mathrm{(H1)}$, $\mathrm{(A1)}$, $\mathrm{(A2)}$ hold. Denote by $\pi^{(1)}=(\pi_i^{(1)})$ and $\tilde{ \pi}^{(1)}=(\tilde{\pi}_i^{(1)})$ respectively the invariant probability measure of $(q_{ij}^{(1)})$ and $(\tilde{q}_{ij}^{(1)})$. Suppose that one of $\mathrm{(L1)}$ and $\mathrm{(L2)}$ holds.
\begin{itemize}
  \item[$1^\circ$] If one of $\sum_{i\in\S}\pi_i^{(1)}\beta_i<0$ and $\sum_{i\in\S}\tilde{\pi}_i^{(1)}\beta_i <0$ holds, and $\lim_{x\to 0}\rho(x)=\infty$, then the equilibrium point $x=0$ is unstable in probability.
  \item[$2^\circ$] If $\sum_{i\in\S}\pi_i^{(1)}\beta_i<0$ and $\sum_{i\in\S}\tilde{\pi}_i^{(1)} \beta_i<0$, and $\rho(x)$ vanishes only at $0$, then the equilibrium point $x=0$ is asymptotically stable in probability.
\end{itemize}
\end{mythm}

\begin{proof}
  $1^\circ$ According to the definition, it is enough to show that for any $\veps>0$, it holds $\lim_{x\to 0^+} \p\big(\sup_{t\geq 0} |X_t^{x,i}|>\veps\big)>0$ or $\lim_{x\to 0^-}\p\big(\sup_{t\geq 0} |X_t^{x,i}|>\veps\big)>0$.  Let $\delta>0$ satisfying $\{x; |x|<\delta\}\subset D$. For $0<\veps<|x|$, define the stopping times
  \[\tau_\veps=\inf\{t>0;|X_t^{x,i}|\leq \veps\}, \quad \eta_{\delta}=\inf\{t>0;|X_t^{x,i}|\geq \delta\}.\]
  We shall only prove this theorem in the situation (L2) holds, and the situation that (L1) holds can be proved with suitable modification as those in Theorem \ref{thm-3}.

  Suppose $\sum_{i\in\S}\pi_i^{(1)} \beta_i<0$, then there exist $c>0$, ${\bm \xi}=(\xi_i)\geq 0$ such that
  \[Q^{(1)}{\bm \xi}(i)=-c-\beta_i,\quad i\in\S.\]
  Define  $V(y,j)=\rho(y)+\xi_jh(y)$, then
  \begin{align*}
    \mathscr{A}V(y,j)&=\mathcal{L}^{(j)}\rho(y)+ \xi_j\mathcal{L}^{(j)} h(y) + Q^{(1)}{\bm \xi}(j) h(y)\\
    &\leq \Big(-c+\xi_j\frac{\mathcal{L}^{(j)}h(y)} {h(y)}\Big)h(y),\quad 0< y< \min \{\delta, \alpha_1\}, \ j\in \S.
  \end{align*}
  As $\lim\limits_{y\to 0} \frac{\mathcal{L}^{(j)} h(y)}{h(y)}=0$, we can take $\delta $ small enough so that   $\forall\,0<\!y\!<\!\delta\!<\!\alpha_1$,
  $ \mathcal{L}^{(j)} h(y) /{h(y)}<\frac c2$,
  and further
   \begin{equation}\label{f-3}
   \mathscr{A}V(y,j)\leq -\frac c2h(y)\leq 0,\quad 0<y< \delta, \ j\in \S.
   \end{equation}
   It follows from It\^o's formula that
   \begin{equation}\label{f-4}
   \E\big[V(X^{x,i}(t\wedge \tau_\veps\wedge \eta_\delta),\La^{x,i}(t\wedge \tau_\veps\wedge \eta_\delta))\big]=V(x,i)+\E\Big[\int_0^{t\wedge \tau_\veps\wedge \eta_\delta}\!\!\mathscr{A}V(X_s^{x,i},\La_s^{x,i}) \d s\Big].\end{equation}
   Here, write $X^{x,i}(t)=X_t^{x,i}$. Due to Lemma \ref{lem-3}(i), as $x>0$, it holds $\p\big(X_s^{x,i}>0,\ \forall\,s\geq 0\big)=1$. By virtue of \eqref{f-3} and the definition of $\eta_\delta$, it follows from \eqref{f-4} by letting $t\to\infty$ that
   \begin{equation}\label{f-5}
   \begin{aligned}
     V(x,i)&\geq \E\big[V(X^{x,i}(\tau_\veps\wedge \eta_\delta),\La^{x,i}(\tau_\veps\wedge \eta_\delta))\big]\\&
     \geq \E\big[ V(X^{x,i}(\tau_\veps),\La^{x,i}(\tau_\veps)) \mathbf1_{\tau_\veps<\eta_\delta}\big]\\
     &\geq \inf_{0<y<\veps,j\in\S}\!\!V(y,j) \, \p(\tau_\veps<\eta_\delta)\\
     &=\inf_{0<y<\veps,j\in\S}\!\!V(y,j) \, \p\big(\sup_{0<t\leq \tau_\veps}\!|X_t^{x,i}|<\delta\big).
   \end{aligned}
   \end{equation}
   It follows from Lemma \ref{lem-3}(i) that $\lim_{\veps\to 0}\tau_\veps=\infty$, a.s. Indeed, let $A=\{\omega;\tau_0(\omega):=\lim_{\veps\to0} \tau_\veps(\omega) <\infty\}$. For $\omega\in A$, as $|X_{\tau_\veps}^{x,i}(\omega)|=\veps$ and $\tau_\veps\downarrow \tau_0$, we get $|X_{\tau_0}^{x,i}(\omega)|=0$. If $\p(A)>0$, then $\p\big(X_t^{x,i}=0 \ \text{for some $t\geq 0$}\big)\geq \p(A)>0$, which contradicts Lemma \ref{lem-3}(i).

   By the boundedness of $(\xi_i)$, $\lim_{x\to 0}\frac{h(x)}{\rho(x)}=0$ and $\lim_{x\to 0}\rho(x)=\infty$,  we obtain that
   \[\lim_{\veps\to 0} \inf_{0<y<\veps,j\in\S} V(y,j)=\lim_{\veps\to 0}\inf_{0<y<\veps,j\in\S}\rho(y)\Big(1+\xi_j \frac{h(y)}{\rho(y)}\Big)=\infty.\]
   Invoking \eqref{f-5}, this means that
   \[\p\big(\sup_{t\geq 0}|X_t^{x,i}|<\delta\big)=0.\]
   Thus, the equilibrium point $x=0$ is unstable in probability by definition.

   $2^\circ$ The basic idea of the proof is similar to that of part $1^\circ$ by using the argument of Lemma \ref{lem-3}(ii) (cf. \cite[Lemma 7.5, Lemma 7.6]{YZ}).
   We still only consider the situation that (L2) holds.

   Since $\sum_{i\in\S}\pi_i^{(1)}\beta_i<0$ and $\sum_{i\in\S}\tilde \pi^{(1)}_i\beta_i<0$, by the Fredholm alternative, there exist a constant $c>0$, vectors ${\bm\xi}$ and $\tilde{\bm{\xi}}$ such that
   \[Q^{(1)}{\bm \xi}(i)=-c-\beta_i,\quad \wt Q^{(1)}{\tilde{\bm\xi}}(i)=-c-\beta_i.\]
   According to the initial value $(X_0,\La_0)=(x,i)$, define
   \begin{equation*}
   V(y,j)=\begin{cases}
     \rho(y)+\xi_jh(y) &\text{if $(y,j)\in (0,\infty)\times\S$},\\
     \rho(y)+\tilde\xi_jh(y)&\text{if $(y,j)\in (-\infty,0)\times\S$}.
   \end{cases}
   \end{equation*}
   Take $\delta >0$ such that $\delta<\alpha_1\wedge (-\alpha_{-1})$ and $\{x; x<\delta\}\subset D$.
   Due to Lemma \ref{lem-3}(i), when $\delta>0$ is small enough, it holds $\mathscr{A}V(X_s^{x,i},\La_s^{x,i})\leq 0$ for $s\leq t\wedge \eta_\delta$ a.s. and hence
   \begin{align*}
   V(x,i)&\geq \E\big[V(X^{x,i}(t\wedge \eta_\delta), \La^{x,i}(t\wedge \eta_\delta) )\big]\\
   &\geq \lambda_\delta \p(\eta_\delta\leq t),
   \end{align*}
   where $\lambda_\delta:=\inf\{V(y,j); |y|>\delta,\,j\in\S\}>0$. Note that $\eta_\delta\leq t$ if and only if $\sup_{0\leq s\leq t} |X_s^{x,i}|>\delta$. Therefore, it follows that
   \begin{equation}\label{f-6}
   \p\big(\sup_{0\leq s\leq t}|X_s^{x,i}|>\delta\big)\leq \frac{V(x,i)}{\lambda_\delta}.
   \end{equation}
   From $\lim_{x\to 0}\rho(x)=\lim_{x\to 0}h(x)/\rho(x)=0$, we get $\lim_{x\to 0}V(x,i)=0$.
   Then, it follows from \eqref{f-6}  that
   \[\lim_{x\to 0} \p\big(\sup_{0\leq s\leq t}|X_s^{x,i}|>\delta\big)=0,\]
   which implies $x=0$ is stable in probability. Furthermore, we can follow the argument of \cite[Lemma 7.6]{YZ} to conclude $x=0$ is asymptotically stable in probability with suitable modification as above, whose details are omitted. The proof is complete.
\end{proof}

\begin{myrem}\label{rem-2}
Note that in the previous argument of Theorem \ref{thm-4}, $2^\circ$, we cannot use directly Lemma \ref{lem-3}(ii), because the function $V$ is defined depending on the initial value $(x,i)$ and we used the property that $X_t^{x,i}$ possesses the same sign  as $x$ for all $t>0$ almost surely. This also makes it difficult to extend Theorem \ref{thm-4} to the higher dimensional space.
\end{myrem}

\begin{myexam}\label{exam-1}
Consider the non-linear hybrid system $(X_t,\La_t)\in\R\!\times\!\S$ satisfying
\begin{equation}\label{f-7}
\d X_t=b_{\La_t} \mathrm{sgn} (X_t)\big(|X_t|^p\!\wedge\!|X_t|\big) \d t+\sigma_{\La_t}\big(|X_t|^q\!\wedge\!|X_t|\big)\d B_t,
\end{equation}
where $p,q>1$ satisfying $p\leq 2q-1$,  $\mathrm{sgn}(x)=1$ if $x\geq 0$; $=-1$ if $x<0$.  $(\La_t)$ is a jumping process on $\S=\{1,2,\ldots,N\}$ satisfying \eqref{a-2} with $(q_{ij}(x))$ given by
\begin{equation}\label{f-8}
q_{ij}(x)=q_{ij}^{(1)}\mathbf1_{[0,\alpha_1)}(x)\!+ \!q_{ij}^{(2)}\mathbf1_{[\alpha_1,\infty)}(x)\!+\! \tilde q_{ij}^{(1)}\mathbf1_{(\alpha_{-1},0)}(x)\!+\!\tilde q_{ij}^{(2)}\mathbf1_{(-\infty,\alpha_{-1}]}(x),
\end{equation}
where $(q_{ij}^{(1)})$, $(q_{ij}^{(2)})$, $(\tilde q_{ij}^{(1)})$, and $(\tilde q_{ij}^{(2)})$ are all conservative, irreducible $Q$-matrices on $\S$; $-\infty<\alpha_{-1}<0<\alpha_1<\infty$. Denote by $(\pi_i^{(1)})$ and $(\tilde\pi_i^{(1)})$ the invariant probability measure  of $(q_{ij}^{(1)})$ and $(\tilde q_{ij}^{(1)})$ respectively. Let
\begin{equation*}
  \beta_i=\begin{cases}
    b_i, &p<2q-1,\\
    b_i-\frac 12\sigma_i^2, &p=2q-1,
  \end{cases} \quad \ \ i\in \S.
\end{equation*} Then,
\begin{itemize}
  \item[$(1)$] If $\sum_{i\in\S}\pi_i^{(1)}\beta_i<0$ and $\sum_{i\in\S}\tilde \pi_i^{(1)}\beta_i<0$, then the equilibrium point $x=0$ is asymptotically stable in probability.
  \item[$(2)$] If $\max\Big\{\sum_{i\in\S}\pi_i^{(1)}\beta_i, \sum_{i\in\S}\tilde \pi_i^{(1)}\beta_i\Big\}>0$, then $x=0$ is unstable in probability.
\end{itemize}
\end{myexam}
We can prove the assertions in Example \ref{exam-1} using Theorem \ref{thm-4} by taking suitable functions $\rho(x)$ and $h(x)$. A little precisely,
for assertion (1), we take $\rho(x)=|x|^{\gamma}$ with $\gamma>0$, and $h(x)=\gamma|x|^{\gamma+p-1}$. One can check directly that
\[\lim_{x\to 0} \frac{h(x)}{\rho(x)}=0,\quad \ \lim_{x\to 0}\frac{\mathcal{L}^{(i)}h(x)} {h(x)}=0,\quad i\in\S.\]
We shall use the arbitrariness of $\gamma>0$ to obtain $\beta_i$ given in this example.
For assertion (2), we take $\rho(x)=|x|^{-\gamma} $ with $\gamma>0$, $h(x)=|x|^{p-\gamma-1}$. One can refer to \cite[Theorem 2.8]{SX} for more details.

\subsection{Criteria on ergodicity and transience}

In this subsection, we shall investigate the recurrent properties for hybrid systems with piecewise constant type switching. We aim to extend the criteria established in \cite{Sh15a} for Markovian regime-switching processes to the current situation.

Consider the hybrid system $(X_t,\La_t)$ satisfying \eqref{a-1}, \eqref{a-2} whose generator $\mathscr{A}$ is given in \eqref{f-0}. Similar to the study in the previous subsection, we shall use the following two Lyapunov type conditions to characterize the behavior of the hybrid system in each fixed state $i\in\S$.
\begin{itemize}
  \item[$\mathrm{(L3)}$] There exist a positive  function $\rho\in C^2(\R^d)$, a constant $r_0>0$, $\beta_i\in\R$ for $i\in\S$, such that
      \[\mathcal{L}^{(i)}\rho(x)\leq \beta_i\rho(x),\quad \ |x|>r_0,\ i\in\S.\]
  \item[$\mathrm{(L4)}$] There are two positive functions $\rho,\,h\in C^2(\R^d)$, a constant $r_0>0$, constants $\beta_i\in\R$ for $i\in \S$, such that
      \begin{gather*}
        \mathcal{L}^{(i)}\rho(x)\leq \beta_ih(x),\quad \ |x|>r_0,\ i\in\S,\\
        \lim_{|x|\to \infty}\frac{h(x)} {\rho(x)}=0,\ \lim_{|x|\to \infty} \frac{\mathcal{L}^{(i)}h(x)}{h(x)} =0.
      \end{gather*}
\end{itemize}

\begin{mythm}\label{thm-5}
Let $(X_t,\La_t)$ be a hybrid system satisfying \eqref{a-1}, \eqref{a-2}, \eqref{a-3}. Suppose $\mathrm{(A1)}$, $\mathrm{(A2)}$ hold. Let $(\pi_i^{(m+1)})$ be the invariant probability measure of $Q^{(m+1)}=(q_{ij}^{(m+1)})$. Assume that one of $\mathrm{(L3)}$ and $\mathrm{(L4)}$ holds, and
\[\sum_{i\in\S} \pi_i^{(m+1)}\beta_i<0.\]
Then $(X_t,\La_t)$ is ergodic if $\lim_{|x|\to \infty} \rho(x)=\infty$; is transient if $\lim_{|x|\to \infty} \rho(x)=0$.
\end{mythm}

\begin{proof} Let us follow the line of \cite[Theorem 2.1,Theorem 3.1]{Sh15a} to prove this theorem.

\noindent\textbf{Case 1}. Suppose (L3) holds. Let $r_1\!>\!\max\{r_0,\alpha_m\}$ and $Q_p^{(m+1)}\!:=Q^{(m+1)}\!+\!p\,\mathrm{diag}({\bm \beta})$ for $p>0$. Since $\sum_{i\in\S}\pi_i^{(m+1)}\beta_i<0$, due to \cite{Bar}, there exists $p_0>0$ such that
for any $0<p<p_0$,
\[-\eta_p:=\max\{\mathrm{Re}\,\gamma;\,\gamma\in \text{spectrum of $Q_p^{(m+1)}$}\}<0,\] and a corresponding eigenvector ${\bm\xi}=(\xi_i)\gg 0$. Let $V(x,i)=\xi_i\rho(x)^p$, then
\begin{align*}
  \mathscr{A}V(x,i)&=Q(x){\bm\xi}(i)\rho(x)^p+\xi_i \mathcal{L}^{(i)}\rho(x)^p\\
  &\leq \big(Q^{(m+1)}{\bm \xi}(i)+p\beta_i\xi_i\big)\rho(x)^p\\
  &=-\eta_p\xi_i\rho(x)^p=-\eta_pV(x,i), \quad \ i\in\S,\  |x|>r_1.
\end{align*}
Therefore, according to the Lyapunov drift condition(cf. \cite{PP}), $(X_t,\La_t)$ is exponentially ergodic (hence, ergodic) if $\lim_{|x|\to \infty} \rho(x)=\infty$; is transient if $\lim_{|x|\to\infty}\rho(x)=0$.

\noindent\textbf{Case 2}. Suppose (L4) holds. Due to the Fredholm alternative, $\sum_{i\in\S}\pi_i^{(m+1)}\!\beta_i<0$ implies that there are $c>0$ and ${\bm \xi}=(\xi_i)\gg 0$ such that
\[Q^{(m+1)}\!{\bm \xi}(i)=\sum_{j\neq i}q_{ij}^{(m+1)}\!(\xi_j-\xi_i)=-c-\beta_i,\quad i\in\S.\]
Put $V(x,i)=\rho(x)+\xi_i h(x)$. Then, as $\lim_{|x|\to\infty}\frac{\mathcal{L}^{(i)}h(x)} {h(x)}=0$, there exists $r_1>\max\{r_0,\alpha_m\}$ with $r_0$ given in (L4) such that
\[\xi_i\frac{\mathcal{L}^{(i)} h(x)}{h(x)} <\frac{c}{2},\quad i\in \S,\ |x|>r_1,\]
and hence
\begin{align*}
  \mathscr{A}V(x,i)&\leq \Big(Q^{(m+1)}{\bm \xi} (i)+\beta_i+\xi_i\frac{\mathcal{L}^{(i)} h(x)}{h(x)}\Big) h(x)\\
  &\leq (-c+\frac c 2)h(x)\leq 0,\quad \ |x|>r_1, \ i\in\S.
\end{align*}
Using the Lyapunov drift condition(cf. \cite{PP}), $(X_t,\La_t)$ is ergodic if $\lim_{|x|\to \infty}\rho(x)=\infty$; is transient if $\lim_{|x|\to \infty} \rho(x)=0$. The proof is completed.
\end{proof}

We go to deal with the 1-dimensional case. Consider the stochastic hybrid system $(X_t,\La_t)$ in $\R\times \S$ satisfying \eqref{a-1}, \eqref{a-2} and \eqref{a-4}. To determine the ergodic property of $(X_t,\La_t)$, based on above discussion, $(q_{ij}^{(0)})$ and $(q_{ij}^{(m+1)})$ will play an important role.

\begin{mythm}\label{thm-6}
Let $(X_t,\La_t)$ be a hybrid system in $\R\!\times\!\S$ satisfying \eqref{a-1}, \eqref{a-2} and \eqref{a-4}. Suppose that $\mathrm{(A1)}$, $\mathrm{(A2)}$ hold and one of $\mathrm{(L3)}$ and $\mathrm{(L4)}$ holds. Let $(\pi_i^{(0)})$ and $(\pi_i^{(m+1)})$ be the invariant probability measure of $(q_{ij}^{(0)})$ and $(q_{ij}^{(m+1)})$ respectively.
\begin{itemize}
  \item[(1)] If $\max\big\{\sum\limits_{i\in\S}\pi_i^{(0)} \beta_i,  \sum\limits_{i\in\S} \pi_i^{(m+1)}\beta_i\big\} <0$, and $\lim_{|x|\to \infty} \rho(x)=\infty$, then $(X_t,\La_t)$ is ergodic.
  \item[(2)] If (i) $\sum\limits_{i\in\S} \pi_i^{(0)}\beta_i<0$ and $\lim_{x\to -\infty} \rho(x)=0$, or (ii) $\sum\limits_{i\in\S}\pi_i^{(m+1)}\beta_i<0$ and $\lim_{x\to \infty} \rho(x)=0$, then $(X_t,\La_t)$ is transient.
\end{itemize}
\end{mythm}

\begin{proof}
  We only prove this theorem under the assumption that (L4) holds. The case that (L3) holds can be proved by modifying the choice of $V(x,i)$ as that in Theorem \ref{thm-5}, and hence the details are omitted.

  (1) Since $\sum_{i\in\S}\pi_i^{(0)}\beta_i<0$ and $\sum_{i\in\S}\pi_i^{(m+1)}\beta_i<0$, the Fredholm alternative yields that there exist $c>0$, ${\bm \xi}=(\xi_i)\gg 0$, $\tilde {\bm \xi}=(\tilde\xi_i)\gg 0$, such that
  \[Q^{(0)}\tilde{\bm \xi}(i)=-c-\beta_i,\quad Q^{(m+1)}{\bm \xi} (i)=-c-\beta_i,\quad i\in\S.\]
  It is clear that there exists a function $V\in C^2(\R\times\S)$ satisfying
  \begin{equation}\label{h-1}
  V(x,i)=\begin{cases}
    \rho(x)+\tilde\xi_ih(x), & x<\alpha_0-1,\\
    \rho(x)+\xi_i h(x), &x>\alpha_m+1,
  \end{cases}\qquad i\in\S.
  \end{equation}
  As $\lim_{|x|\to \infty}\frac{\mathcal{L}^{(i)}h(x)}{h(x)} =0$, there exists $r_1>\max\{|\alpha_0-1|,|\alpha_m+1|\}$ such that
  \[\xi_i\frac{\mathcal{L}^{(i)} h(x)}{h(x)}<\frac c2,\quad \text{and}\ \ \tilde \xi_i\frac{\mathcal{L}^{(i)}h(x)}{h(x)} <\frac c2,\quad |x|>r_1,\, i\in\S.
  \]
  Hence,
  \begin{equation}\label{h-2}
  \begin{split}
  \mathscr{A}V(x,i)&=\begin{cases} \mathcal{L}^{(i)}\rho(x)+ \xi_i\mathcal{L}^{(i)}h(x)+ Q^{(m+1)}{\bm \xi}(i) h(x), &x>r_1,\\
  \mathcal{L}^{(i)}\rho(x)+ \tilde \xi_i\mathcal{L}^{(i)}h(x)+ Q^{(m+1)}{\tilde{\bm \xi}}(i) h(x), &x<-r_1,
  \end{cases}
  \\
  &\leq -\frac c2h(x)<0, \quad |x|>r_1,\ i\in\S.
  \end{split}
  \end{equation}
  Using the Lyapunov drift condition, we get $(X_t,\La_t)$ is ergodic.

  (2) We consider the situation (ii) $\sum_{i\in\S}\pi_i^{(m+1)}\beta_i<\infty$ and $\lim_{x\to \infty}\rho(x)=0$, and the case (i) can be proved in a similar way.

  The Fredholm alternative yields from $\sum_{i\in\S}\pi_i^{(m+1)}\beta_i<0$ that there exists $c>0$ and ${\bm\xi}=(\xi_i)\gg 0$ such that
  \[Q^{(m+1)}{\bm \xi}(i)=-c-\beta_i,\quad \quad i\in\S.\]
  The fact $\lim_{|x|\to \infty}\frac{\mathcal{L}^{(i)}h(x)} {h(x)}=0$ tells us the existence of $r_1>\max\{r_0,\alpha_m\}$ such that
  \begin{equation}\label{h-3}
  \xi_i\frac{\mathcal{L}^{(i)}h(x)}{h(x)}<\frac c2,\quad \ \forall\,x>r_1,\ i\in\S.
  \end{equation}
  Let $V(x,i)=\rho(x)+\xi_ih(x)$, then
  \begin{equation}\label{h-4}
  \mathscr{A}V(x,i)\leq \Big(\beta_i +Q^{(m+1)}{\bm \xi}(i) +\xi_i\frac{\mathcal{L}^{(i)}h(x)}{h(x)} \Big) h(x)
  \leq -\frac c 2h(x)<0,\quad i\in\S, \ x>r_1.
  \end{equation}
  Since $\lim_{|x|\to \infty} \frac{h(x)}{\rho(x)}=0$, there exist $x_0,\,\kappa_1,\,\kappa_2>r_1$ such that $r_1<\kappa_1<x_0<\kappa_2$, and
  \begin{equation}\label{h-5}
    \rho(\kappa_1)+\big(\inf_{i\in\S}\xi_i\big) h(\kappa_1)>\rho(x_0)+\xi_{i_0} h(x_0)\quad \text{for some $i_0\in\S$}.
  \end{equation}
  Taking the initial point $(X_0^{x_0,i_0},\La_0^{x_0,i_0})=(x_0,i_0)$, define the stopping times
  \[\tau_{\kappa_1}=\inf\{t>0; X_t^{x_0,i_0}<\kappa_1\},\quad \eta_{\kappa_2}=\inf\{t>0; X_t^{x_0,i_0}>\kappa_2\}.\]
  By Dynkin's formula,
  \begin{equation*}
  \begin{split}
    \E V\big(X^{x_0,i_0}(t\!\wedge \! \tau_{\kappa_1}\!\wedge\! \eta_{\kappa_2}), \La^{x_0,i_0}(t\!\wedge \! \tau_{\kappa_1}\!\wedge\! \eta_{\kappa_2})\big)&=V(x_0,i_0)\!+\!\E\Big[ \!\int_0^{ t \wedge   \tau_{\kappa_1}\!\wedge  \eta_{\kappa_2} } \!\!\!\! \mathscr{A}V(X_s^{x_0,i_0}, \La_s^{x_0,i_0})\d s\Big]\\
    &\leq V(x_0,i_0).
    \end{split}
  \end{equation*}
  Letting $t\to \infty$, it follows
  \[V(x_0,i_0)\geq \inf_{i\in\S}V(\kappa_1,i)\p (\tau_{\kappa_1}<\eta_{\kappa_2}) +\inf_{i\in\S} V(\kappa_2,i) \p(\eta_{\kappa_2}<\tau_{\kappa_1}).\]
  Then, due to \eqref{h-5},
  \begin{equation}\label{h-6}
  \p(\tau_{\kappa_1}>\eta_{\kappa_2})\geq
  \frac{\inf_{i\in\S}V(\kappa_1,i)-V(x_0,i_0)} {\inf_{i\in\S}V(\kappa_1,i)-\inf_{i\in\S} V(\kappa_2,i)}.
  \end{equation}
  The linear growth condition (A1) implies that $\lim_{\kappa_2\to\infty} \eta_{\kappa_2}=\infty$ almost surely. Letting $\kappa_2\to \infty$, noting $\lim_{\kappa_2\to \infty} V(\kappa_2,i)=0$, we have
  \begin{equation}\label{h-7}
  \p(\tau_{\kappa_1}=\infty)\geq \frac{\inf_{i\in\S} V(\kappa_1,i)-V(x_0,i_0)}{\inf_{i\in\S} V(\kappa_1,i)}>0.
  \end{equation}
  Hence, the process $(X_t^{x_0,i_0},\La_t^{x_0,i_0})$ is transient, and the proof is complete.
\end{proof}

\begin{myexam}\label{exam-2}
Consider the Ornstein-Uhlenbeck type process $(X_t,\La_t)$ satisfying
\begin{equation}\label{a-5}
\d X_t=b_{\La_t}X_t\d t+\sigma_{\La_t}\big(X_t^2\wedge |X_t|\big)\d B_t,
\end{equation} and \eqref{a-2}
with  $(q_{ij}(x))_{i,j\in\S}$ given by
\begin{equation}\label{a-6}
q_{ij}(x)=q_{ij}^{(1)}\mathbf1_{[0,\alpha_1)}(x)\!+\! \tilde q_{ij}^{(1)} \mathbf1_{(\alpha_{-1},0)}(x)\!+ \!q_{ij}^{(2)}\mathbf1_{[\alpha_1,\infty)}(x)  \!+\!\tilde q_{ij}^{(2)}\mathbf1_{(-\infty,\alpha_{-1}]}(x),
\end{equation}
where $(q_{ij}^{(1)})$, $(q_{ij}^{(2)})$, $(\tilde q_{ij}^{(1)})$, and $(\tilde q_{ij}^{(2)})$ are all conservative, irreducible $Q$-matrices on $\S$; $ \alpha_{-1}<0<\alpha_1 $
Then, $(X_t,\La_t)$ is exponentially ergodic if $\max\big\{\sum_{i\in\S} \pi_i^{(2)}\mu_i,\sum_{i\in \S} \tilde \pi_i^{(2)}\mu_i\big\}<0$; $(X_t,\La_t)$ is transient if $\max\big\{\sum_{i\in\S} \pi_i^{(2)}\mu_i,\sum_{i\in \S} \tilde \pi_i^{(2)}\mu_i\big\}>0$.
\end{myexam}

These assertions can be proved by using Theorem \ref{thm-6}. Indeed, there is $\rho\in C^2(\R)$ such that $\rho(x)=|x|$ for $|x|\geq r_0>0$. Then, for each $i\in\S$,
\[\mathcal{L}^{(i)} \rho(x)=\mu_i \rho(x),\quad |x|\geq r_0,\]
and hence $(X_t,\La_t)$ is exponentially ergodic if $\max\big\{\sum_{i\in\S} \pi_i^{(2)}\mu_i,\sum_{i\in \S} \tilde \pi_i^{(2)}\mu_i\big\}<0$ by virtue of Theorem \ref{thm-6} (1).

To justify the transience, let $\rho(x)=|x|^{-\gamma}$ with $\gamma>0$ for $|x|>r_0>0$. Then,
\[\mathcal{L}^{(i)} \rho(x)\leq (-\gamma \mu_i+\frac1{r_0})\rho(x),\quad |x|>r_0,\ i\in\S.\]
If $\max\big\{\sum_{i\in\S} \pi_i^{(2)}\mu_i,\sum_{i\in \S} \tilde \pi_i^{(2)}\mu_i\big\}>0$, there exists $r_0>0$ large enough such that
\[\max\Big\{-\gamma \sum_{i\in\S} \pi_i^{(2)}\mu_i+\frac1{r_0},-\gamma \sum_{i\in \S} \tilde \pi_i^{(2)}\mu_i+\frac1{\gamma_0} \Big\}<0.\]
Since $\lim_{|x|\to \infty} \rho(x)=0$, it follows from Theorem \ref{thm-6} (2) that $(X_t,\La_t)$ is transient in this case.

In the end of this work, we provide a result on the ergodicity of state-dependent hybrid systems with $(q_{ij}(x))$ being Lipschitz continuous in $x$, which can simplify greatly the investigation of ergodicity for certain class of state-dependent regime-switching processes.

Let $Q(x)=(q_{ij}(x))$ be a conservative, irreducible $Q$-matrix on $\S$ for every $x\in \R^d$, and
\[|q_{ij}(x)-q_{ij}(y)|\leq K_4|x-y|,\quad x,\,y\in\R^d,\ i,\,j\in\S,\]
for some $K_4>0$. Assume (A1) holds. Let $(X_t,\La_t)$ be a hybrid system satisfying \eqref{a-1}, \eqref{a-2}.

\begin{mythm}\label{thm-7}
Assume that the limit
\begin{equation}\label{f-9}
\lim_{|x|\to\infty} Q(x)=Q
\end{equation} exists and $Q$ is still irreducible. Denote by $(\pi_i)$ the invariant probability measure of $Q$. Suppose $\mathrm{(L3)}$ or $\mathrm{(L4)}$ holds, and
\begin{equation}\label{f-10}
\sum_{i\in\S}\pi_i\beta_i<0.
\end{equation} Then $(X_t,\La_t)$ be expenentially ergodic if $\lim_{|x|\to\infty} \rho(x)=\infty$; is transient if $\lim_{|x|\to \infty} \rho(x)=0$.
\end{mythm}

\begin{proof}
  We prove this theorem under the assumption that (L4) holds. The case (L3) holds can be proved by using the Perron-Frobenius theorem similar to the proof of Theorem \ref{thm-3}.

According to the Fredholm alternative, due to $\sum_{i\in\S}\pi_i\beta_i<0$, we obtain that there exist $c>0$ and ${\bm \xi}=(\xi_i)\gg 0$ such that
\[Q{\bm\xi}(i)=-c-\beta_i,\quad \ i\in\S.\]
It follows from \eqref{f-9}, $\lim_{|x|\to\infty}\frac{\mathcal{L}^{(i)}h(x)} {h(x)}=0$ that there exists $r_1>r_0 $ with $r_0$ given in (L4) such that
\[|Q(x){\bm\xi}(i)-Q{\bm\xi}(i)|\leq \frac{c}{4},\quad \xi_i\frac{\mathcal{L}^{(i)}h(x)}{h(x)}<\frac c4,\quad \,|x|>r_1, \ i\in\S.\]
Taking $V(x,i)=\rho(x)+\xi_ih(x)$, we get
\begin{align*}
\mathscr{A}V(x,i)&\leq \Big(\beta_i+Q(x){\bm \xi}(i)+\xi_i \frac{\mathcal{L}^{(i)}h(x)}{h(x)}\Big)h(x)\\
&\leq \big(\beta_i+Q{\bm\xi}(i)+\frac c2\big)h(x)\\
&\leq -\frac c 2h(x)<0,\quad \quad  \, |x|>r_1,\ i\in\S.
\end{align*}
Consequently, the desired conclusion follows from the Lyapunov drift condition (cf. \cite{PP}).
\end{proof}

\end{document}